\documentclass[11pt,reqno,a4paper]{amsart}
\usepackage[utf8]{inputenc}
\usepackage{amsmath,pdfsync,verbatim,graphicx,epstopdf,enumerate,latexsym,amssymb}
\usepackage[margin=1in]{geometry}

\usepackage[usenames]{color}
\usepackage[abbrev,nobysame,alphabetic]{amsrefs}
\usepackage{mathtools}  
\usepackage{thmtools}
\usepackage{thm-restate}
\mathtoolsset{showonlyrefs}

\usepackage[colorlinks=true]{hyperref}
\usepackage{cancel}
\usepackage{accents}
\usepackage[framemethod=tikz]{mdframed}
\hypersetup{linkcolor=blue,citecolor=blue}

\usepackage{amsmath,pdfsync,verbatim,graphicx,epstopdf,enumerate}

\usepackage{mathtools}
\usepackage{amsthm}
\newtheorem{theorem}{Theorem}[section]
\newtheorem{lemma}[theorem]{Lemma}
\newtheorem{proposition}[theorem]{Proposition}

\newcommand{\Rn}{\mathbb{R}^n}

\newcommand{\D}{\mathrm{d}}
\newcommand{\Lc}{\mathcal{L}}

\newcommand{\Dc}{\mathcal{D}}

\def\tre{\textcolor{red}}

\usepackage{amsmath,pdfsync,verbatim,graphicx,epstopdf,enumerate}

\newcommand{\Sc}{\mathcal{S}}

\newcommand{\lv}{\lVert}
\newcommand{\rv}{\rVert}

\newcommand{\I}{\mathrm{i}}
\newcommand{\err}{\mathbf{r}}
\renewcommand{\O}{\Omega}
\newcommand{\PD}{\partial}
\newcommand{\abs}[1]{|#1|}

\title{Linearized partial data  Calder{\'o}n problem for Biharmonic operators}

\author{Divyansh Agrawal}
\address{Centre for Applicable Mathematics, Tata Institute of Fundamental Research, India.}
\email{agrawald@tifrbng.res.in}
\author{Ravi Shankar Jaiswal}
\address{Centre for Applicable Mathematics, Tata Institute of Fundamental Research, India.}
\email{ravi@tifrbng.res.in}
\author{Suman Kumar Sahoo}
\address{Department of Mathematics, ETH Z\"urich, Z\"urich, Switzerland}
\email{susahoo@ethz.ch}

\begin{document}
\begin{abstract}
We consider a linearized partial data Calder{\'o}n problem for biharmonic operators extending the analogous result for harmonic operators \cite{linearized_partial_data}. We construct special solutions and utilize Segal-Bargmann transform to recover  lower order perturbations.
\end{abstract}

\maketitle
\section{Introduction and main result}

The Calder\'on problem, introduced by Calder\'on in $1980$ \cite{Calderon1980}, aims to recover the electrical conductivity $\gamma$ of a medium based on the Dirichlet to Neumann map (DN map) denoted as $ \Lambda_{\gamma}:=\gamma \PD_{\nu} u|_{\PD\Omega}$. Here, $u$ represents the solution to the conductivity equation $ \nabla \cdot(\gamma\nabla u)=0$ with specified Dirichlet boundary conditions. As the mapping from $\gamma$ to $\Lambda_{\gamma}$ is nonlinear, it is valuable to investigate its linearization, known as the linearized Calder\'on problem. In \cite{Calderon1980}, he solved the linearized problem by proving the following result that if $\gamma$ is a bounded function in $\Omega$ and satisfies the equation \begin{align*}
    \int_{\Omega} \gamma \nabla u\cdot \nabla v=0 \quad \mbox{for all harmonic functions $u$ and $v$ in $ \Omega$, then $ \gamma=0$ in $\Omega$.}
\end{align*}

Sylvester and Uhlmann solved the  Calder\'on problem for $C^2(\Omega)$ conductivities in \cite{SYL}. Subsequently, various authors have investigated inverse problems for different types of partial differential equations. For more comprehensive results on the subject, we refer to the surveys \cites{Uhlmann_survey,Uhl_eip_survey}.

A closely related inverse problem is for the Schr\"odinger equation $ (-\Delta+q)u=0$, which can be derived from the conductivity equation for $C^2$ conductivities \cite{SYL}. For this reason the linearized problem for the Schr\"odinger is also an interesting question and 
 it can be formulated in the following way: Let $q$ be a bounded function in $\Omega$ and satisfies 
\begin{align*}
     \int_{\Omega} quv=0 \quad  \mbox{for all harmonic functions $u$ and $v$ in $ \Omega$, then  $ q=0$ in $ \Omega$.}
\end{align*}

Another closely related  problem involves recovering the conductivity $ \gamma$ (or potential) from partial measurements. In  \cite{Kenig_annals_2007}, the authors proved the unique determination of $q$ when the Dirichlet and Neumann measurements are given on two complementary open subsets of the boundary (roughly speaking) in dimensions $n\ge 3$. However, the unique recovery of $q$, when the Dirichlet and Neumann data are prescribed on the same part of the boundary, remains an open question in dimensions $n\ge 3$.  In two dimensions, this was solved by Imanuvilov, Uhlmann, and Yamamoto \cite{Imanuvilov_JAMS_partial}. Several partial answers are known  either by dropping the support condition on the Dirichlet data or by assuming additional symmetry on the domain $ \Omega$; see \cites{Uhlman_Greanleaf_twoplane,BUK,Kenig_annals_2007,Isakov_partial,K_S,Kenig_Salo_Survey}. Nevertheless, for an arbitrary domain $\Omega$, a linearized version of this problem was addressed  in \cite{linearized_partial_data}. For a comprehensive overview of the Calder\'on problem with partial data, we recommend the survey article \cite{Kenig_Salo_Survey}.


This article focuses on a linearized Calder\'on problem for biharmonic operators, inspired by \cite{linearized_partial_data}. 
Let  $\Omega\subset \mathbb{R}^n\, (n\ge 2)$  be a bounded domain with smooth boundary $\PD\O$. Let $ \Sigma\subseteq
\PD\O $ be a non-empty open subset of $ \PD\O$.  We consider the following  boundary value problem for biharmonic operator $\Lc$ with lower order anisotropic perturbations  up to order $3$: 
\begin{equation}\label{operator}
\begin{aligned}
\begin{cases}
    	\Lc(x,D)&=
		(-\Delta)^2 + Q(x,D)\quad  \mbox{in} \quad \Omega
	   \\
	    (u,\PD_{\nu} u)&= (f_1,f_2) \quad \qquad \hspace{9mm} \mbox{on} \quad \PD\Omega
\end{cases}
\end{aligned}
\end{equation}
where $ Q(x,D):= \sum_{l=0}^{3} a^{l}_{i_1\cdots i_{l}}(x) \, D^{i_1\cdots i_l}$	is a differential operator of order $3$ with $1\le i_1,\cdots,i_l\le n$ and $a^{l}$ is a smooth symmetric  tensor field of order $l$ in $\overline{\Omega}$. Here $ f_j$  are suitable functions on the boundary with support of $f_j\subset\Sigma$, for $j=1,2$.  Einstein summation convention is assumed for repeated indices throughout the article. 

Suppose $0$ is not an eigenvalue of \eqref{operator}, then the boundary measurements associated to \eqref{operator} can be encoded in terms of partial DN map as follows:
\begin{align}\label{dn_map}
    \Lambda_{Q} (f_1,f_2):= (\PD^2_{\nu} u|_{\Sigma}, \PD^3_{\nu} u|_{\Sigma}).
\end{align}
Alternatively, one can also prescribe the partial boundary measurements in terms of the Cauchy data set as follows:
\begin{align}
\mathcal{C}_{Q,\Sigma}:= \{(u,\PD_{\nu}u, \PD^2_{\nu}u,\PD^3_{\nu} u)|_{\Sigma}: u\in H^4(\Omega), \Lc (x,D)u=0 \: \mbox{in}\: \Omega\}.
\end{align}
The inverse problem we are interested in is to show that the Fr\'echet derivative of $ \Lambda_{Q}$ (evaluated at $Q=0$) is injective. This result establishes local uniqueness for the linearized Calder\'on problem of biharmonic operators. To state our main result, let us define the set
\begin{align}\label{def_mathcal_E}
   \mathcal{E} := \{u \in C^\infty(\overline{\O}): \Delta^2 u = 0 \, \text{in} \, \O, 
 \, \mbox{and}\, (u, \partial_\nu u)= 0 \, \text{on} \, \Gamma\}, \quad \mbox{where} \quad \Gamma:= \partial \O \backslash \Sigma. 
\end{align}
Our main result is
\begin{theorem}\label{th:main_result}
    Let $a^2,a^0\in C^{\infty}(\overline{\Omega})$ and $a^1\in C^{\infty}(\overline{\Omega}; \Rn)$. Suppose 
\begin{align}\label{integral_identity}
\int\limits_{\Omega}\left(a^2\Delta u  + a^1\cdot \nabla u  + a^0 u\right) v \, \D x = 0 \quad \mbox{ holds for all} \quad u,v \in \mathcal{E}.
\end{align}
    Then  $ a^j=0$ in $\Omega$, for $ j=0,1,2$.
\end{theorem}

Previous research has focused on inverse problems for polyharmonic operators (which are perturbations of $ (-\Delta)^m$ when $m\ge 2$) of order 2m. Krupchyk, Lassas, and Uhlmann initiated the study of inverse problems for polyharmonic operators in their works \cites{KRU1,KRU2}. More recently, several authors have investigated inverse problems for higher-order elliptic operators, such as biharmonic and polyharmonic operators; see \cites{Ghosh-Krishnan,BG_19,BG_biharmonic_second_order}. These works have successfully solved inverse problems for polyharmonic operators which  recovered functions, vector fields, or two tensor fields. 
In a  recent  work \cite{polyharmonic_mrt_application} the authors solved an inverse problem for polyharmonic operators of order $2m$ with lower-order tensorial perturbations up to order $m$, utilizing momentum ray transforms. Furthermore, a linearized Calder\'on problem for polyharmonic operators of order $2m$ with lower order perturbation up-to $2m-1$   was considered in  \cite{SS_linearized}, also employing momentum ray transform techniques. However, when dealing with partial data, the lack of sufficient information to extract the momentum ray transform of unknown tensor fields poses a technical challenge. In this work, instead of relying on the momentum ray transform, the authors employ techniques of the Segal-Bargmann transform 
from the work \cite{linearized_partial_data}. It appears that addressing linearized partial data inverse problems for polyharmonic operators with  lower order anisotropic perturbations may require new techniques or tools, which the authors aim to explore in future research.

To the best of the authors' knowledge, this work is the first to consider the linearized problem for higher-order operators with partial data. Recent research has shown that solutions to the linearized problem can be used to solve inverse problems for nonlinear partial differential equations (PDEs) using higher-order linearization techniques, as demonstrated in previous works. We refer  \cites{LLLS_JMPA,Fractional_power_LLST,Krupchyk_remark,Krupchyk_isotropic_quasilinear,Kian_quasilinear_partial} for nonlinear elliptic PDEs, \cites{KLU_invention,KLUO_duke,Hintz_cpde} for nonlinear hyperbolic PDEs to name a few. The authors hope that their work will pave the way for similar results in the future for higher-order elliptic operators.


The rest of the article is structured as follows.  In Section \ref{sec:preliminaries}, we start by presenting some preliminary results that are essential for the proof of our main result. The section is divided into four subsections, each dealing with a specific topic. In the first subsection, the integral identity is transformed to an equivalent integral identity through a conformal change of variables. The second subsection discusses the construction of special solutions of biharmonic equations, where their Dirichlet data vanishes in a part of the boundary (as shown in Lemma \ref{lm: special solutions}). The third subsection provides a brief introduction to the Segal-Bargmann transform and the fourth subsection presents a local uniqueness result (Proposition \ref{lm: local uniqueness}).
Finally, Section \ref{sec: proof of main result} focuses  on demonstrating the main result of the article which is Theorem \ref{th:main_result}. In appendix \ref{appen_linearization}  we linearize Dirichlet to Neumann map, and in appendix \ref{appendix_decay} we present a decay estimate of special solutions used in Lemma \ref{lm: special solutions}.



\section{Preliminaries}\label{sec:preliminaries}
\subsection{Change of coordinates}Let us first consider a change of coordinates for ease of calculations. Fix a point $x_0 \in \Sigma$ and choose an exterior ball to $\O$ at $x_0$, say $B(a,r)$ i.e. $\overline{\O} \cap \overline{B(a, r)} = \{x_0\}$. Consider the conformal change of variables
\[
\psi: x \mapsto \frac{x-a}{|x-a|^2}r^2 + a
\]
which fixes the point $x_0$  and maps $\O$ to the interior of the ball $B(a,r)$.  Next, we observe that a function $u$ is biharmonic if and only if $u^* = r^{n - 4}|x - a|^{4 - n} u \circ \psi$ is biharmonic. This is the analog of the Kelvin transform for the bilaplacian; see \cite{kelvin-biharmonic}.

A function $u$ and normal derivative of $u$
are zero on $\Gamma$ if and only if $u^*$ and normal derivative of $u^{*}$
are zero on $\psi(\Gamma)$.
Consequently, the integral identity becomes 
\[
\int\limits_{\psi (\O)} \left(\tilde{a}^2\Delta u  + \tilde{a}^1\cdot \nabla u  + \tilde{a}^0 u\right) v \, \D x = 0,
\]
for all smooth biharmonic functions $u,v$ in $\psi (\O)$ such that $(u, \frac{\partial u}{\partial \nu})|_{\psi(\Gamma)} = 0 = (v, \frac{\partial v}{\partial \nu})|_{\psi(\Gamma)}$.  Moreover, $a^0, a^1$ and $a^2$ are zero near $x_0$ if and only if $\tilde{a}^0, \tilde{a}^1$ and $\tilde{a}^2$ are zero near $x_0$. After a rotation and translation we can bring $x_0$ to the origin. Thus our set-up will be as follows: $\O \subset B(-e_1, 1)$ and $\Gamma \subset \{x_1 < -2c\}$ for some $c>0$, after a suitable translation and rotation. We want to show the following:  
\[
\int\limits_{\Omega} [a^2 \Delta u + a^1_i \partial_i u + a^0 u ] v \, \mathrm{d} x = 0 \quad \mbox{for all} \quad u,v \in \mathcal{E} \implies a^2= a^1=a^0 = 0  \quad \mbox{in} \quad \Omega.
\]
\subsection{Construction of special solutions} In this section we carry out the construction of special solutions of biharmonic operators. To construct the solutions which vanish on part of the boundary, we use the method analogous to \cite{linearized_partial_data}. To this end, let $\chi \in C_c^\infty(\mathbb{R}^n)$ be a cut-off function which is 1 in a neighbourhood of $\Gamma$, and define $
H_K(y):= \sup\limits_{x \in  K} x \cdot y$, where $K = \mathrm{supp} \, \chi \cap \partial \Omega$ which can be taken to be subset of $\{x_1 < -c\}$.

\begin{lemma}\label{lm: special solutions}
   Suppose $\xi \in \mathbb{C}^n $ such that $\xi \cdot \xi = 0$ and  $a$ be any smooth function satisfying 
    \begin{align*}
        \Delta a=constant \quad \mbox{and} \quad \sum_{i,j}\nabla_{ij}^2 a\cdot \xi_i\xi_j= 0.
    \end{align*}
    There exists $u \in \mathcal{E}$ of the form $  u= e^{-\I x\cdot \xi/h} a(x) +r(x,h),$
    where $r$ satisfies 
    \begin{align*}
     \|r\|_{H^2(\Omega)} \leq C (1 + \frac{|\xi|^2}{h^2} + \frac{|\xi|^4}{h^4})^{1/2} e^{\frac{1}{h} H_K (\mathrm{Im} \ \xi)},
     \end{align*}
     where  $C$ is a constant independent of  $\xi$, and $h$. 
\end{lemma}

\begin{proof}
    Fix a smooth function $a$ satisfying the assumptions of the lemma. For such $a$ and $\xi$, the function $ae^{-\mathrm{i} x \cdot \xi / h}$ is biharmonic in $\Omega$. The term $r$ serves as the correction term which forces the solution to vanish on the required part of the boundary. Let us choose $r$ to be the solution of 
    \begin{equation}
        \begin{cases}
            \Delta^2 r &= 0 \quad \mbox{in} \quad  \Omega \\
            (r, \partial_\nu r) &= (-(ae^{-\mathrm{i} x \cdot \xi / h}) \chi, - \partial_\nu (ae^{-\mathrm{i} x \cdot \xi / h}) \chi) \quad \mbox{on} \quad \partial \Omega.
        \end{cases}
    \end{equation}
    Clearly, the function $u= e^{-\I x\cdot \xi/h} a(x) +r(x,h) \in \mathcal{E}$ and the bounds on $r$ are obtained from Lemma \ref{lem:decay}.
\end{proof}

\subsection{Segal-Bargmann transform}\label{sec:Segal-Bargmann transform}
The Segal-Bargmann transform  \cite{linearized_partial_data} of a function $f$ on $\mathbb{R}^n$ is defined for $z \in \mathbb{C}^n$ as
\[
Tf(z):= \int e^{-\frac{1}{2h} (z-y)^2} f(y) \,\mathrm{d}y.
\]
This is well-defined for $f \in L^\infty(\mathbb{R}^n)$ and one has
\begin{align*}
    |Tf(z)| \leq \int |e^{-\frac{1}{2h} (z-y)^2}| |f(y)| \, \mathrm{d}y 
     \leq \int e^{-\frac{1}{2h} (|\mathrm{Re}z - y|^2 - |\mathrm{Im} z|^2)} |f(y)|\, \mathrm{d}y  \leq e^{\frac{1}{2h} |\mathrm{Im} z|^2} \|f\|_\infty (2 \pi h)^{n/2}.
\end{align*}
The property most important to us is that when $z \in \mathbb{C}^n$ is restricted to $x \in \mathbb{R}^n$, we obtain 
$
\frac{1}{(2 \pi h)^{n/2}} T f (x) = f \ast G (x),$
which is convolution of $f$ with Gaussian $G(x)= e^{-\frac{|x|^2}{2h}}$ . 
Therefore, for a bounded function $f$ with compact support, 
\[
\lim_{h \to 0} \frac{1}{(2 \pi h)^{n/2}} \, T f = f \quad \text{in} \,\, L^p(\mathbb{R}^n) \,\, \text{for all} \,\, 1 \leq p < \infty.
\]
Therefore, to prove that $f$ vanishes in a region, we will prove that the limit on the left vanishes in that region. To do this, we will utilize the exponential bounds of special solutions constructed in the previous section; see Lemma \ref{lm: special solutions}. 
Notice that 
\[
e^{-\frac{1}{2h} (z-y)^2} = e^{-\frac{z^2}{2h}} (2 \pi h)^{-n/2} \int\limits_{\mathbb{R}^n} e^{-\frac{t^2}{2h}} e^{-\frac{\mathrm{i}}{h} y \cdot (t + \mathrm{i} z)} \mathrm{d}t.
\]
Therefore interchanging the order of integration, we obtain 
\begin{align*}
    |Tf(z)| &\leq \frac{1}{(2 \pi h)^{n/2}} \, e^{\frac{-1}{2h} (|\mathrm{Re} z|^2 - |\mathrm{Im}z|^2)} \int\limits_{\mathbb{R}^n} e^{-\frac{t^2}{2h}} \left |\int\limits_{\mathbb{R}^n} e^{-\frac{\mathrm{i}}{h} y \cdot (t + \mathrm{i} z)} f(y) \mathrm{d}y \right | \mathrm{d}t \\
    & \leq \frac{1}{(2 \pi h)^{n/2}} \, e^{\frac{-1}{2h} (|\mathrm{Re} z|^2 - |\mathrm{Im}z|^2)} \Big( \int\limits_{|t| < \epsilon a} + \int\limits_{|t| \geq \epsilon a} \Big ) \left ( e^{-\frac{t^2}{2h}} \left |\int\limits_{\mathbb{R}^n} e^{-\frac{\mathrm{i}}{h} y \cdot (t + \mathrm{i} z)} f(y) \mathrm{d}y \right | \right ).
\end{align*}
For functions $f$ supported in the bounded set $\Omega \subset \{ y_1 \leq 0\}$, the above estimate yields
\begin{equation}\label{intermediate-estimate}
    |Tf(z)| \leq e^{\frac{-1}{2h} (|\mathrm{Re} z|^2 - |\mathrm{Im}z|^2)} \left(\sup_{|t| < \epsilon a} \left | \int e^{-\frac{\mathrm{i}}{h} y \cdot (t + \mathrm{i} z)} f(y) \mathrm{d}y \right | + \sqrt{2} e^{\frac{1}{h} |\mathrm{Re} z'|} e^{- \frac{\epsilon^2 a^2}{4h}} \int\limits_\Omega |f(y)| \mathrm{d}y \right).
\end{equation}
\subsection{A local uniqueness result}

\begin{proposition}\label{lm: local uniqueness}
    Let $ \Omega\subset \Rn\, (n\ge 2)$ be a bounded domain with smooth boundary. Let  $ x_0\in \partial\Omega \setminus \Gamma$. Under the assumptions of Theorem \ref{th:main_result}, there exists $ \delta >0$ such that $ a^j=0$ in $ B(x_0,\delta) \cap \Omega$, for $j=0,1,2$. 
\end{proposition}

\begin{proof}
We start with following integral identity 
\begin{equation}\label{integra_equation}
\int\limits_{\Omega} (a^2 \Delta u  + a_i^1 \partial_i u  + a^0 u) v \, \mathrm{d}x = 0,
\end{equation}
where $\O \subset B(-e_1, 1)$ and $\Gamma \subset \{x_1 < -2c\}$ for some $c>0$.
The idea of the proof is as follows: we use the special solutions constructed in Lemma \ref{lm: special solutions} to derive some identities, using which we will be able to establish the desired estimates. These estimates when invoked in \eqref{intermediate-estimate} will then imply that the coefficients vanish locally. We will do this for each of the three coefficients in the following steps.\smallskip

\textbf{Step 1}: We prove that $a^2=0$ near origin. \smallskip  

To achieve this let us first choose $ u(x,\xi,h) = u_\sharp(x,\xi,h) = |x|^2 e^{\frac{- \mathrm{i} x \cdot \xi}{h}} +\err$ 
and $v(x,\eta,h) = v_0(x,\eta,h) = e^{\frac{- \mathrm{i} x \cdot \eta}{h}} + \err_1$ and insert them into \eqref{integra_equation} to obtain
\begin{equation}\label{eq5*}
\begin{split}
    0 &= \int\limits_\Omega \Big[a^2 \{ 2n e^{\frac{-\mathrm{i} x \cdot \xi}{h}} - 4 \frac{\mathrm{i}}{h} x \cdot \xi e^{\frac{-\mathrm{i} x \cdot \xi}{h}} + \Delta \err  \}  + a_i^1 \{ 2 x_i e^{\frac{-\mathrm{i} x \cdot \xi}{h}} - \frac{\mathrm{i}}{h} \xi_i |x|^2 e^{\frac{-\mathrm{i} x \cdot \xi}{h}} + \partial_i \err\} \\ & \qquad + a^0 \{|x|^2 e^{\frac{- \mathrm{i} x \cdot \xi}{h}} + \err\}\Big] (e^{\frac{- \mathrm{i} x \cdot \eta}{h}} +\err_1) \, \mathrm{d}x.
\end{split}
\end{equation}
Let us again choose $u(x,\xi,h) = u_j(x,\xi,h) = x_j e^{\frac{- \mathrm{i} x \cdot \xi}{h}} + \tilde{\err}_j $ and $v(x,\eta,h) = x_j e^{\frac{- \mathrm{i} x \cdot \eta}{h}} + \tilde{\err}^1_j $ and sum over $j$ to conclude
\begin{equation}\label{eq6}
\begin{aligned}
    0 &= \int\limits_\Omega  \Big[ a^2 \{- \frac{ 2 \mathrm{i}}{h} \xi_j e^{\frac{- \mathrm{i} x \cdot \xi}{h}} + \Delta \tilde{\err}_j\} + a_i^1 \{\delta_{ij} e^{\frac{- \mathrm{i} x \cdot \xi}{h}} - \frac{\mathrm{i}}{h} \xi_i x_j e^{\frac{- \mathrm{i} x \cdot \xi}{h}} + \partial_i \tilde{\err}_j\}\\& \qquad + a^0 \{x_j e^{\frac{- \mathrm{i} x \cdot \xi}{h}} + \tilde{\err}_j\}\Big] (x_j e^{\frac{- \mathrm{i} x \cdot \eta}{h}} + \tilde{\err}^1_j)\D x.
\end{aligned}
\end{equation}
Subtracting \eqref{eq5*} from twice of \eqref{eq6} we obtain
 \begin{equation}
     \begin{aligned}\label{eq7}
    0 &= \int\limits_{\Omega} [2na^2 e^{\frac{- \mathrm{i} x \cdot (\xi + \eta)}{h}} + 2n a^2\err_1 e^{\frac{- \mathrm{i} x \cdot \xi}{h}} - 4 \frac{\mathrm{i}}{h} x \cdot \xi a^2\err_1 e^{\frac{- \mathrm{i} x \cdot \xi}{h}} + 4 \frac{\mathrm{i}}{h} a^2 \xi_j \tilde{\err}^1_j e^{\frac{- \mathrm{i} x \cdot \xi}{h}} + a^2 \Delta \err (e^{\frac{- \mathrm{i} x \cdot \eta}{h}} +\err_1) \\
    & \qquad -2  a^2 \Delta \tilde{\err}_j (x_j e^{\frac{- \mathrm{i} x \cdot \eta}{h}} + \tilde{\err}^1_j)+ 2 a_j^1 x_j e^{\frac{- \mathrm{i} x \cdot \xi}{h}}\err_1 - 2 a_j^1 \tilde{\err}^1_j e^{\frac{- \mathrm{i} x \cdot \xi}{h}} + \frac{\mathrm{i}}{h} a_j^1 \xi_j |x|^2 e^{\frac{- \mathrm{i} x \cdot (\xi + \eta)}{h}} - \frac{\mathrm{i}}{h} a^1_i \xi_i |x|^2\err_1 e^{\frac{- \mathrm{i} x \cdot \xi}{h}}  \\
    &\qquad+ 2 \frac{\mathrm{i}}{h} a_i^1 \xi_i x_j \tilde{\err}^1_j e^{\frac{- \mathrm{i} x \cdot \xi}{h}} + a_i^1 \partial_i \err (e^{\frac{- \mathrm{i} x \cdot \eta}{h}} +\err_1) - 2 a_i^1 \partial_i \tilde{\err}_j (x_j e^{\frac{- \mathrm{i} x \cdot \eta}{h}} + \tilde{\err}^1_j) - a^0 |x|^2 e^{\frac{- \mathrm{i} x \cdot (\xi + \eta)}{h}}   \\
    & \qquad+  a^0\err_1 |x|^2 e^{\frac{- \mathrm{i} x \cdot \xi}{h}}+ a^0 \err (e^{\frac{- \mathrm{i} x \cdot \eta}{h}} +\err_1) - 2 a^0 x_j \tilde{\err}^1_j e^{\frac{- \mathrm{i} x \cdot \xi}{h}} - 2 a^0 \tilde{\err}_j (x_j e^{\frac{- \mathrm{i} x \cdot \eta}{h}} + \tilde{\err}^1_j)] \, \mathrm{d}x. 
    \end{aligned}
\end{equation}
Finally, taking $u(x,\xi,h) = u_0 (x,\xi,h) = e^{\frac{- \mathrm{i} x \cdot \xi}{h}} + \err_2$ and $ v(x, \eta, h) = v_\sharp (x, \eta, h) = |x|^2 e^{\frac{- \mathrm{i} x \cdot \eta}{h}} + \err_3$, we  conclude
\begin{align}\label{eq8}
    0 &= \int\limits_\Omega  [ a^2 \{\Delta \err_2\} + a_i^1 \{- \frac{\mathrm{i}}{h} \xi_i  e^{\frac{- \mathrm{i} x \cdot \xi}{h}} + \partial_i \err_2\} + a^0 \{ e^{\frac{- \mathrm{i} x \cdot \xi}{h}} + \err_2\}] (|x|^2 e^{\frac{- \mathrm{i} x \cdot \eta}{h}} + \err_3) \, \mathrm{d}x.
\end{align}
We add \eqref{eq7} 
 and \eqref{eq8} to derive%
\begin{align*}
    \int\limits_{\Omega} a^2 e^{\frac{- \mathrm{i} x \cdot (\xi + \eta)}{h}} \mathrm{d}x & = -\frac{1}{2n} \Bigg[ \int\limits_{\Omega}  2n a^2\err_1 e^{\frac{- \mathrm{i} x \cdot \xi}{h}} - 4 \frac{\mathrm{i}}{h} x \cdot \xi a^2\err_1 e^{\frac{- \mathrm{i} x \cdot \xi}{h}} + 4 \frac{\mathrm{i}}{h} a^2 \xi_j \tilde{\err}^1_j e^{\frac{- \mathrm{i} x \cdot \xi}{h}} + a^2 \Delta \err (e^{\frac{- \mathrm{i} x \cdot \eta}{h}} +\err_1) \\
    &\qquad+ a^2 \Delta \err_2 (|x|^2 e^{\frac{- \mathrm{i} x \cdot \eta}{h}} + \err_3) -2  a^2 \Delta \tilde{\err}_j (x_j e^{\frac{- \mathrm{i} x \cdot \eta}{h}} + \tilde{\err}^1_j)+ 2 a_j^1 x_j e^{\frac{- \mathrm{i} x \cdot \xi}{h}}\err_1 - 2 a_j^1 \tilde{\err}^1_j e^{\frac{- \mathrm{i} x \cdot \xi}{h}} \\
    &\qquad-\frac{\mathrm{i}}{h} a_j^1 \xi_j e^{\frac{- \mathrm{i} x \cdot (\xi)}{h}} \err_3  +a^1_i \partial_i \err_2 (|x|^2 e^{\frac{- \mathrm{i} x \cdot \eta}{h}} + \err_3) - \frac{\mathrm{i}}{h} a_i^1 \xi_i |x|^2\err_1 e^{\frac{- \mathrm{i} x \cdot \xi}{h}} + 2 \frac{\mathrm{i}}{h} a_i^1 \xi_i x_j \tilde{\err}^1_j e^{\frac{- \mathrm{i} x \cdot \xi}{h}} \\
    &\qquad+ a_i^1 \partial_i \err (e^{\frac{- \mathrm{i} x \cdot \eta}{h}} +\err_1) - 2 a_i^1 \partial_i \tilde{\err}_j (x_j e^{\frac{- \mathrm{i} x \cdot \eta}{h}} + \tilde{\err}^1_j) + a^0 \err_3 e^{\frac{-\mathrm{i} x \cdot \xi}{h}}+ a^0 \err_2 (|x|^2 e^{\frac{- \mathrm{i} x \cdot \eta}{h}} + \err_3) \\
    & \qquad+ a^0\err_1 |x|^2 e^{\frac{- \mathrm{i} x \cdot \xi}{h}} + a^0 \err (e^{\frac{- \mathrm{i} x \cdot \eta}{h}} +\err_1) - 2 a^0\, x_j \tilde{\err}^1_j e^{\frac{- \mathrm{i} x \cdot \xi}{h}} - 2 a^0\, \tilde{\err}_j (x_j e^{\frac{- \mathrm{i} x \cdot \eta}{h}} + \tilde{\err}^1_j) \, \mathrm{d}x \Bigg].
\end{align*}
The goal of the exercise was that now on the right hand side, each summand contains at least one term which has a good decay, as stated in Lemma~\ref{lm: special solutions}. Thus, we have the bound: 
\begin{align}\label{est-xi+eta}
    |\int\limits_{\Omega} a^2 e^{\frac{- \mathrm{i} x \cdot (\xi + \eta)}{h}} \mathrm{d}x| & \leq C (\|a^2\|_\infty + \|a^1\|_\infty + \|a\|_\infty) (1 + \frac{|\xi|^2}{h^2} + \frac{|\xi|^4}{h^4})^{1/2} (1 + \frac{|\eta|^2}{h^2} + \frac{|\eta|^4}{h^4})^{1/2}  \\
    & \hspace{2cm} \times  e^{\frac{1}{h} H_K (\mathrm{Im} \ \xi)} e^{\frac{1}{h} H_K (\mathrm{Im} \ \eta)}.
\end{align}

We use the decomposition result, stated in \cite{DKSU}, which says that any $z \in \mathbb{C}^n$ such that $|z-2\mathrm{i}ae_1| < 2\epsilon a$, for $\epsilon$ small enough, can be written as 
\begin{equation}
    z = \xi + \eta, \quad  \xi \cdot \xi = 0 = \eta \cdot \eta, \quad |\xi - a (\mathrm{i}e_1 + e_2)| < C \epsilon a, \quad |\eta + a(e_2 - \mathrm{i}e_1)| < C \epsilon a.
\end{equation}
Notice that for $\epsilon$ small enough, the vectors $\xi$ and $\eta$ can be chosen such that $\mathrm{Im} \xi_1 , \mathrm{Im} \eta_1 \geq \frac{a}{4}$. Keeping in mind that $\Omega \subset B(-e_1, 1)$ and $K \subset \{x_1 < -c\}$, this yields
\begin{equation}
    H_K(\mathrm{Im} \xi) \leq -c \mathrm{Im} \xi_1  + |\mathrm{Im} (\xi')|, \quad \mbox{and} \quad   H_K(\mathrm{Im} \eta) \leq -c \mathrm{Im} \eta_1  + |\mathrm{Im} (\eta')|.
\end{equation}
 This decomposition and estimate invoked in \eqref{est-xi+eta} leads to 
\begin{equation}\label{estimate-integral}
    \abs{\int\limits_{\Omega} a^2 e^{\frac{- \mathrm{i} x \cdot z}{h}} \mathrm{d}x}  \leq C h^{-4} (\|a^2\|_\infty + \|a^1\|_\infty + \|a^0\|_\infty) e^{-\frac{ca}{2h}} e^{\frac{2C\epsilon a}{h}},
 \end{equation}
 for all $z \in \mathbb{C}^n$ such that $|z- 2 \mathrm{i} a e_1| < 2 \epsilon a$. 
When $|t| < \epsilon a$ and $|z-2ae_1| < \epsilon a$, we also have $|(t + \mathrm{i} z) - 2\mathrm{i}ae_1| < 2 \epsilon a$. We next insert \eqref{estimate-integral} in \eqref{intermediate-estimate} to estimate  $Ta^2$, the Segal-Bargmann transform of $a^2$ introduced in Section \ref{sec:Segal-Bargmann transform}. 
\begin{align*}
    |Ta^2(z)| \leq C h^{-4} (\|a^2\|_\infty + \|a^1\|_\infty + \|a^0\|_\infty) e^{\frac{-1}{2h} (|\mathrm{Re} z|^2 - |\mathrm{Im}z|^2)} (e^{-\frac{ca}{2h}} e^{\frac{2C \epsilon a}{h}} + e^{-\frac{\epsilon^2 a^2}{4h}} e^{\frac{\epsilon a}{h}}),
\end{align*}
whenever $|z - 2ae_1| < \epsilon a$. Now choosing $\epsilon < c/8C$ and $a > (c+ 4 \epsilon)/\epsilon^2$, we deduce
\[
|Ta^2(z)| \leq C h^{-4} (\|a^2\|_\infty + \|a^1\|_\infty + \|a^0\|_\infty) e^{\frac{1}{2h} (|\mathrm{Im}z|^2 -|\mathrm{Re} z|^2 - \frac{ca}{2})}.
\]
Therefore, we end up with the following bound, as in \cite{linearized_partial_data}:
\begin{align*}
    e^{-\frac{\Phi(z_1)}{2h}} |Ta^2(z_1, x')| \leq C & h^{-4} (\|a^2\|_\infty + \|a^1\|_\infty + \|a^0\|_\infty) \times
      \begin{cases}
        1, \quad \text{when} \quad z_1 \in \mathbb{C} \\
        e^{-\frac{ca}{4h}},\quad \text{when}\,\, |z_1 - 2a| \leq \frac{\epsilon a}{2}, |x'| < \frac{\epsilon a}{2}
    \end{cases}
\end{align*}
where $ x' \in \mathbb{R}^{n-1}$ and the weight $\Phi$ is defined as:
\[
\Phi(z_1) = \begin{cases}
    (\mathrm{Im} z_1)^2,  &\text{when} \, \mathrm{Re} z_1 \leq 0 \\
    (\mathrm{Im} z_1)^2 - (\mathrm{Re} z_1)^2, &\text{when} \, \mathrm{Re} z_1 \geq 0.
\end{cases}
\]
Now we are in a position to invoke \cite{linearized_partial_data}*{Lemma~4.1} for the function 
\[
F(s) = \frac{h^4 \, Ta^2(s, x')}{C (\|a^2\|_\infty + \|a^1\|_\infty + \|a^0\|_\infty)},
\]
to obtain that there exists $c' > 0$ such that
\[
|Ta^2(x)| \leq C h^{-4} (\|a^2\|_\infty + \|a^1\|_\infty + \|a^0\|_\infty) e^{- \frac{c'}{2h}},
\]
for all $x \in \Omega$ such that $|x_1| \leq \delta$ for $\delta$ small enough. Multiplying this estimate by $(2 \pi h)^{-n/2}$ and letting $h \to 0$, we obtain 
$
a^2(x) = 0 $ for $ x \in \Omega,\, |x_1| < \delta$. This implies $a^2=0$ near origin. \smallskip

\textbf{Step 2}: We show that $a^1=0$ near origin. \smallskip

Now we have the identity 
\[
\int\limits_{\Omega} (a^2 \Delta u  + a^1_i \partial_i u  + a^0 u) v \, \mathrm{d}x = 0, \quad\mbox{with $a^2(x) = 0$\quad for \,$ x \in \Omega,$\, such that\, $   |x_1| < \delta$}.
\]
For fixed $j$, let us now choose the solutions $u (x, \xi, h) = u_j (x, \xi, h) = x_j e^{\frac{- \mathrm{i} x \cdot \xi}{h}} + \tilde{\err}_j $ and $v (x, \eta, h) = v_0 (x, \eta, h) = e^{\frac{- \mathrm{i} x \cdot \eta}{h}} +\err_1$ in the above identity to obtain 
\begin{align}\label{eq10}
    \hspace{-5mm}\int\limits_{\Omega}  \Big[a^2\{\frac{- 2 \mathrm{i} \xi_j}{h} e^{\frac{- \mathrm{i} x \cdot \xi}{h}} + \Delta \tilde{\err}_j\} + a_i^1 \{(\delta_{ij}  - \frac{\mathrm{i} x_j \xi_i}{h} ) e^{\frac{- \mathrm{i} x \cdot \xi}{h}} + \partial_i \tilde{\err}_j\} + a^0 \{x_j e^{\frac{- \mathrm{i} x \cdot \xi}{h}} + \tilde{\err}_j\} \Big] (e^{\frac{- \mathrm{i} x \cdot \eta}{h}} +\err_1) \, \mathrm{d}x = 0.
\end{align}
Now, for the same $j$ as above, we again choose $u (x, \xi, h) = u_0(x, \xi, h) = e^{\frac{- \mathrm{i} x \cdot \xi}{h}} + \err_2$ and $v(x, \eta, h) = v_j (x, \eta, h) = x_j e^{\frac{- \mathrm{i} x \cdot \eta}{h}} + \tilde{\err}^1_j$ to obtain
\begin{align}\label{eq11}
    \int\limits_{\Omega} \Big[a^2 \Delta \err_2 + a_i^1 \{- \frac{\mathrm{i} \xi_i}{h} e^{\frac{- \mathrm{i} x \cdot \xi}{h}} + \partial_i \err_2 \} + a^0 \{ e^{\frac{- \mathrm{i} x \cdot \xi}{h}} + \err_2 \}\Big] (x_j e^{\frac{- \mathrm{i} x \cdot \eta}{h}} + \tilde{\err}^1_j) \, \mathrm{d}x = 0.
\end{align}
Subtracting equation \eqref{eq11} from \eqref{eq10}, we deduce
\begin{align*}
    &\int\limits_{\Omega} a^1_j e^{\frac{-\mathrm{i} x \cdot (\xi + \eta)}{h}} \\ &= \int\limits_{\Omega} a^2 (\frac{ 2 \mathrm{i} \xi_j}{h} - \Delta \tilde{\err}_j)(e^{\frac{- \mathrm{i} x \cdot \eta}{h}} +\err_1) + a^2 \Delta \err_2 (x_j e^{\frac{- \mathrm{i} x \cdot \eta}{h}} + \tilde{\err}^1_j) + \int\limits_{\Omega} -a^1_i (\partial_i \tilde{\err}_j)(e^{\frac{- \mathrm{i} x \cdot \eta}{h}} +\err_1) \\
    & \qquad+ \int\limits_{\Omega} a^1_i (\partial_i \err_2) (x_j e^{\frac{- \mathrm{i} x \cdot \eta}{h}} + \tilde{\err}^1_j) - a^1_j e^{\frac{- \mathrm{i} x \cdot \xi}{h}}\err_1 +\int\limits_{\Omega} a^1_i \frac{\mathrm{i} x_j \xi_i}{h}\err_1 - a^1_i \frac{\mathrm{i} \xi_i}{h} e^{\frac{- \mathrm{i} x \cdot \xi}{h}} \tilde{\err}^1_j \, \mathrm{d}x -a^0 \tilde{\err}_j (e^{\frac{- \mathrm{i} x \cdot \eta}{h}} +\err_1) \\ 
    & \qquad+ \int\limits_{\Omega}  a^0 \err_2 (x_j e^{\frac{- \mathrm{i} x \cdot \eta}{h}} + \tilde{\err}^1_j) - a^0 x_j e^{\frac{- \mathrm{i} x \cdot \xi}{h}}\err_1 + a^0 e^{\frac{- \mathrm{i} x \cdot \xi}{h}} \tilde{\err}^1_j \, .
\end{align*}
Notice that except for the first term in the integrand, every term can be bounded using the decay bound on the remainder terms. To bound the first term, we use the fact that $a^2(x) = 0$ for $x_1 >- \delta$, and obtain
\begin{align}
    \left |\int\limits_\Omega a^2 \frac{2\mathrm{i}\xi_j}{h} e^{\frac{-\mathrm{i}x \cdot \eta}{h}} \, \mathrm{d}x \right| &= \left |\:\int\limits_{\Omega \cap \{x_1 \leq -\delta\}} a^2 \frac{2\mathrm{i}\xi_j}{h} e^{\frac{-\mathrm{i}x \cdot \eta}{h}} \, \mathrm{d}x \right|
    \leq C \frac{\|a^2\|_\infty}{h} \times  e^{-\frac{\delta a}{4h}} \,\times e^{|\mathrm{Im}\eta'| / h}.
\end{align}
Now estimating as in \textbf{Step 1}, we obtain 

\begin{equation}
    \abs{\int\limits_{\Omega} a^1_j e^{\frac{- \mathrm{i} x \cdot z}{h}} \mathrm{d}x}  \leq C h^{-4} (\|a^2\|_\infty + \|a^1\|_\infty + \|a^0\|_\infty) e^{-\frac{\delta a}{2h}} e^{\frac{2C\epsilon a}{h}},
 \end{equation}
 for all $z \in \mathbb{C}^n$ such that $|z- 2 \mathrm{i} a e_1| < 2 \epsilon a$. 
Carrying out exactly as in \textbf{Step 1}, we find that $a^1_j(x) = 0$ for $x \in \Omega$ such that $|x_1| < \delta'$. \smallskip

\textbf{Step 3}: Finally we show that $a^0=0$ near origin.\smallskip

We have the identity 
\[
\int\limits_{\Omega} (a^2 \Delta u  + a_i^1 \partial_i u  + a^0 u) v \, \mathrm{d}x = 0,
\]
with $a^2(x) = 0 = a^1(x) \quad \text{for} \,\, x \in \Omega, |x_1| < \delta'$.
Now using the solutions $u(x, \xi, h) = u_0(x, \xi, h) = e^{\frac{- \mathrm{i} x \cdot \xi}{h}} + \err_2$ and $v(x,\eta, h) = v_0 (x, \eta, h) = e^{\frac{- \mathrm{i} x \cdot \eta}{h}} +\err_1$, and estimating as above, we arrive at 

\begin{equation}
    \abs{\int\limits_{\Omega} a^0 e^{\frac{- \mathrm{i} x \cdot z}{h}} \mathrm{d}x}  \leq C h^{-4} (\|a^2\|_\infty + \|a^1\|_\infty + \|a^0\|_\infty) e^{-\frac{\delta' a}{2h}} e^{\frac{2C\epsilon a}{h}},
 \end{equation}
 for all $z \in \mathbb{C}^n$ such that $|z- 2 \mathrm{i} a e_1| < 2 \epsilon a$.

Again, proceeding as in \textbf{Step 1 and Step 2}, we conclude $a^0(x) = 0$ for $x \in \Omega, |x_1| < \delta''$ for $\delta'' > 0$ small enough. Thus combining all the steps, we obtain  that $a^2, a^1$ and $a^0$ vanish in a small enough neighbourhood of the origin.
\end{proof}

\section{Proof of main result}\label{sec: proof of main result}
Our objective in this section is to demonstrate that $a^2, a^1$ and $a^0$ 
 vanish identically. To achieve this, we first present the following density result. 
\begin{lemma}\label{lm:density}
    Let $\Omega_1 \subset \Omega_2$ be two bounded open sets with smooth boundaries. Let $G_{2}$ be the Green kernel for the biharmonic operator associated to the open set $\Omega_2$ $i.e.,$
    \begin{equation*}
        \begin{split}
    -\Delta_y^2G_{2}(x, y) &= \delta(x - y) \quad \mbox{for all} \quad x,y\in \Omega_2\\
        \Big (G_{2}(x, \cdot), \frac{\partial G_{2}}{\partial \nu}(x, \cdot)\Big )&= 0 \quad\qquad\quad \mbox{on}\quad \PD\Omega_2.  
        \end{split}
    \end{equation*}
Then the set 
 $ \mathcal{A}$ is dense in the subspace of all biharmonic functions $u \in C^{\infty}(\overline{\Omega}_1)$ such that
$u|_{\partial \Omega_1 \cap \partial\Omega_2} = 0 = \frac{\partial u}{\partial \nu}|_{{\partial \Omega_1 \cap \partial\Omega_2}},$
    equipped with $L^2(\O_1)$ topology, where 
\begin{equation}\label{aprox set}
     \mathcal{A}:=   \left\{\int\limits_{\Omega_2} G_{2}(\cdot, y) a(y) \, \mathrm{d}y : a \in C^{\infty}(\overline{\Omega}_2), \operatorname{supp}a \subset \Omega_2 \setminus \overline{\Omega}_1
        \right\}.
    \end{equation}
\end{lemma}
\begin{proof}
    Let $v \in L^2(\Omega_1)$ be a function which is orthogonal to subspace \eqref{aprox set}, then by Fubini's theorem we have
    \[\int\limits_{\Omega_2}a(y)\left(\:\int\limits_{\Omega_1}G_{2}(x, y)v(x) \, \mathrm{d}x\right)\, \mathrm{d}y = 0,\]
    for all $a \in C^{\infty}(\overline{\Omega}_2)$ supported in $\Omega_2 \setminus \overline{\Omega}_1$, therefore
    $\int_{\Omega_1}G_{2}(x, y)v(x) \, \mathrm{d}x = 0$
    for all $y \in \Omega_2 \setminus \overline{\Omega}_1$. We want to show that $v$ is orthogonal to any biharmonic function $u \in C^{\infty}(\overline{\Omega}_1)$  satisfying
     $u|_{\partial \Omega_1 \cap \partial\Omega_2} = 0 = \frac{\partial u}{\partial \nu}|_{{\partial \Omega_1 \cap \partial\Omega_2}}.$
    Let us consider the function 
    $
     w(y) = \int_{\Omega_1}G_{2}(x, y)v(x) \, \mathrm{d}x.
     $
     Clearly, $w\in H^4(\Omega_2)$, and it satisfies $ \Delta^2 w=v$ in $\Omega_1$. 
     Then, using the fact that $ \Delta^2 w=v$ in $ \Omega_1$, we obtain
     \begin{equation*}
         \begin{split}
             \int\limits_{\Omega_1} uv \,\mathrm{d}x &=  \int\limits_{\Omega_1}u\Delta^2w - \int\limits_{\Omega_1} w\Delta^2u\\
             &= - \int\limits_{\partial \Omega_1} \partial_{\nu}u \Delta w + \int\limits_{\partial \Omega_1} u \partial_{\nu}(\Delta w) + \int\limits_{\partial \Omega_1} \partial_\nu w \Delta u - \int\limits_{\partial \Omega_1}w\partial_{\nu}(\Delta u
             )\\
             &= \int\limits\limits_{\partial \Omega_1 \setminus \partial \Omega_2}-\partial_{\nu}u \Delta w + u \partial_{\nu}(\Delta w) + \partial_\nu w \Delta u - w\partial_{\nu}(\Delta u
             ).
         \end{split}
     \end{equation*}
    Since $w(y)=\int_{\Omega_1}G_{2}(x, y)v(x) \, \mathrm{d}x = 0$
    for all $y \in \Omega_2 \setminus \overline{\Omega}_1$ and $w\in H^4(\Omega_2)$, this implies  $\PD^k_{\nu} w=0$ on $ \partial \O_1 \backslash \partial \O_2$ for $k=0,\cdots,3$. 
     Hence we conclude
     $\int_{\Omega_1}uv \, \mathrm{d}x = 0,$
     for all $u$ such that $\Delta^2 u = 0,u|_{\partial \Omega_1 \cap \partial\Omega_2} = 0 = \frac{\partial u}{\partial \nu}|_{{\partial \Omega_1 \cap \partial\Omega_2}}$. This completes the proof.
\end{proof}

We now present the  proof of Theorem \ref{th:main_result}. 

\begin{proof}[Proof of Theorem \ref{th:main_result}]
    Fix a point $x_1 \in \Omega$, and let $\Theta: [0,1] \to \overline{\Omega}$ be a smooth curve joining $x_0 \in \partial \Omega \backslash \Gamma$ to $x_1$. The curve $ \Theta$ satisfies  $\Theta(0)=x_0$, $\Theta'(0)$ is the interior normal to $\partial \Omega$ at $x_0$, and $\Theta(t) \in \Omega $ for all $ t \in (0,1]$. By local result (Proposition \ref{lm: local uniqueness}), there exists $ \epsilon_0 > 0$ such that $(a^0,a^1,a^2) = 0$ in $B_{\epsilon_0}(x_0)$. To this end, we consider the closed neighbourhood of the curve ending at $\Theta(t)$  defined as: 
    \begin{equation*}
        \Theta_{\epsilon}(t):= \{x \in \overline{\Omega}: \mathrm{dist}(x, \Theta([0,t])) \leq \epsilon\}.
    \end{equation*}
   Moreover, we define the set
    \begin{equation*}
            I = \{t \in [0,1] : a^0= a^1= a^2=0\, \,\,\text{a.e. on} \,\, \Theta_{\epsilon}(t) \cap \Omega\}.
    \end{equation*}
    Clearly, $I$ is a  closed subset of $[0,1]$. Also, $I$ is non-empty for $\epsilon$ small enough, by the local result. Let us now prove that $I$ is also open. WLOG, assume $\epsilon$ small enough such that $\Theta_\epsilon (1) \subset\joinrel\subset \overline{\Omega}$.
    
    Let $t_0 \in I$ ( which ensures $ [0, t_0] \subset I$). Let us smoothen $\Omega \backslash \Theta_{\epsilon}(t_0)$ into an open subset $\Omega_1$ of $\Omega$ with smooth boundary, such that $\Omega_1 \supset \Omega \backslash \Theta_\epsilon (t_0)$. For $ \epsilon<\epsilon_0$ we have that $\partial \Theta_{\epsilon}(t_0) \cap \partial \Omega \subset B_{\epsilon_0}(x_0) \cap \partial \Omega \subset \partial \Omega \backslash \Gamma$. This implies $\partial \Omega_1 \cap \partial \Omega \supset \Gamma$. Let us also smoothen out $\Omega \cup B_{\tilde{\epsilon}}(x_0) \,\, (\tilde{\epsilon} < \epsilon < \epsilon_0)$ into an open subset $\Omega_2$ with smooth boundary. Hence,
    \begin{equation*}
        \Omega \cup B_{\tilde{\epsilon}}(x_0) \subset \Omega_2, \quad \mbox{and} \quad
    \Gamma\subset \partial \Omega_1 \cap \partial \Omega \subset    \partial \Omega_2 \cap \partial \Omega.
    \end{equation*}
     For each $x\in \Omega_2$, let $G_2(x,\cdot)$ be the Green kernel associated to the open set $\Omega_2$:
    \begin{equation*}
        \begin{cases}
            \Delta^2_y G_2 (x,y) = \delta (x-y) \\
           \left (G_2(x, \cdot)|_{\PD\Omega_2},\frac{\partial G_2}{\partial \nu(\cdot)} (x, \cdot)|_{\partial \Omega_2}\right) = 0.
        \end{cases} 
    \end{equation*}
To proceed further, we next define a new function $G$ as follows:
\begin{align}
    G(x,y):= a^2(y)\, \Delta_y G_2(x,y)  +   \langle \, a^1(y) , \nabla_y G_2(x,y) \, \rangle \ +  a^0(y) \, G_2(x,y) \quad \mbox{where $x\in \Omega_2 \backslash \overline{\Omega}_1, y\in \Omega_1$}.
\end{align}
Then the function 
    \begin{equation*}
   (z,x) \ni  \Omega_2 \backslash \overline{\Omega}_1 \times \Omega_2 \backslash \overline{\Omega}_1 \longmapsto\int\limits_{\Omega_1} G(x,y)\,G_2(z,y)\, \mathrm{d}y, 
    \end{equation*}
    is biharmonic viewed as a function of  $z$ and $x$ variables. Moreover, it satisfies
    \begin{align}\label{indentity_1}
    \int\limits_{\Omega_1} G(x,y)\, G_2(z,y)\mathrm{d}y 
        = \int\limits_{\Omega} G(x,y)\, G_2(z,y)\mathrm{d}y,
    \end{align}
since $(a^0, a^1, a^2)$ vanish on $\Theta_{\epsilon}(t_0) \cap \Omega$, and $\Omega \backslash \overline{\Omega}_1 \subset \Theta_{\epsilon}(t_0) \cap \Omega $. We see that the functions $y\mapsto G_2(z,y)$ and $y \mapsto G_2(x,y)$  are in $C^\infty(\overline{\Omega})$, biharmonic in $\Omega$, and vanish on $\Gamma \subset \partial \Omega_2$, when $z,x \in \Omega_2 \backslash \overline{\Omega}$. Thus $y\mapsto G_2(z,y)$ and $y \mapsto G_2(x,y)$ are in $\mathcal{E}$, see  definition \eqref{def_mathcal_E}. Next   inserting them into the integral identity \eqref{integral_identity} we obtain $ \int_{\Omega} G(x,y)\, G_2(z,y)\, \D y=0$ for all $ z,x \in \Omega_2\setminus\overline{\Omega}$. Note that,  biharmonic functions are real analytic, and any real analytic function that vanishes on some open set must be zero everywhere. The combination of this along with \eqref{indentity_1}, and $ \int_{\Omega} G(x,y)\, G_2(z,y)\, \D y=0$ for all $ z,x \in \Omega_2\setminus\overline{\Omega}$ implies $\int_{\Omega_1} G(x,y)\, G_2(z,y)\, \D y=0$ for all $ z,x \in \Omega_2\setminus\overline{\Omega_1} =0$.   
  Next, multiplying $ \int_{\Omega_1} G(x,y)\, G_2(z,y)\, \D y$  by $b_1 = b_1(z) \in C_c^\infty(\Omega_2 \backslash \overline{\Omega}_1),$ and then taking integrating over $ \Omega_2$  we obtain
    \begin{equation*}
    \int\limits_{\Omega_2} \int\limits_{\Omega_1} G(x,y)\, G_2(z,y)\, \D y \,\, b_1(z) \, \mathrm{d}z = 0,  \,\, \mbox{for all}\; x \in \Omega_2 \backslash \overline{\Omega}_1. 
    \end{equation*}
    Using Fubini theorem, this further entails 
\begin{align}
      \int\limits_{\Omega_1} G(x,y) \left(\:\int\limits_{\Omega_2}  G_2(z,y) \, b_1(z)  \D z\right)\, \mathrm{d}y = 0,  \,\, \mbox{for all}\; x \in \Omega_2 \backslash \overline{\Omega}_1.
\end{align}
We next use the density result from Lemma \ref{lm:density} to approximation any biharmonic function $ u$ in $ \overline{\Omega_1}$ with $ u|_{\partial \Omega_1 \cap \partial\Omega_2} = 0 = \frac{\partial u}{\partial \nu}|_{{\partial \Omega_1 \cap \partial\Omega_2}} $, 
  utilizing the integral of the form  $\int_{\Omega_2}  G_2(z,y) \, b_1(z)  \D z$. This gives
    \begin{equation*}
  \int\limits_{\Omega_1} u(y)\,G(x,y)\, \D y=  \int\limits_{\Omega_1} u(y) [a^2 \Delta_y G_2(x,y) + \langle \, a^1(y) , \nabla_y G_2(x,y) \, \rangle + a^0(y) G_2(x,y)] \, \mathrm{d}y = 0,
    \end{equation*}
  for all $x \in \Omega_2 \backslash \overline{\Omega}_1$.  
    Integration by parts then gives
    \begin{align*}
    0=\int\limits_{\Omega_1} &\Delta_y(a^2(y) u(y)) G_2(x,y) - \langle \nabla u(y) , a^1(y) \rangle G_2(x,y) -  \mathrm{div}(a^1)(y) \, G_2(x,y) \, u(y) \, + \, a^0(y) u(y) G_2(x,y) \, \mathrm{d}y \\
        &+ \int\limits_{\partial \Omega_1} a^2(y)u(y) \frac{\partial G_2}{\partial \nu}(x,y) - G_2(x,y)\frac{\partial (a^2u(y))}{\partial \nu}\, \mathrm{d} \sigma (y) + \int\limits_{\partial \Omega_1} u(y) \langle a^1, \nu \rangle G_2(x,y) \, \mathrm{d} \sigma (y).
    \end{align*}
    Since $a^1 = 0 $ on $\partial \Omega_1 \cap \Omega$, $a^2 = 0$ on $\Omega_1^c \cap \Omega$ and $G_2(x, \cdot) = 0$ on $\partial \Omega_1 \cap \partial \Omega_2$, for all $x \in \Omega_2 \backslash \overline{\Omega}_1$, this implies
    \begin{equation*}
    \begin{split}
    0=&\int\limits_{\Omega_1} [a^2\Delta u 
        (y) + u(y)\Delta a^2 + 2\langle \nabla a^2, \nabla u(y)\rangle] G_2(x,y) - \langle \nabla u(y) , a^1(y) \rangle
        G_2(x,y)\D y\\& \quad-\int\limits_{\Omega_1} (\mathrm{div}(a^1) \, G_2(x,y) \, u(y) - a^0(y) u(y) G_2(x,y) )\, \mathrm{d}y. 
    \end{split}
    \end{equation*}
Multiplying this by a function $b_2 = b_2(x) \in C_c^{\infty}(\Omega_2 \backslash \overline{\Omega}_1)$ and integrating over $\O_2$,  
    \[ \int\limits_{\Omega_2}\int\limits_{\Omega_1}[a^2 \Delta u(y) + \langle \nabla u(y), 
    2\nabla a^2 - a^1(y) \rangle + \left(\Delta a^2 - \operatorname{div}(a^1) + a^0(y)\right)u(y)]G_2(x, y)\, \mathrm{d}y b_2(x)\mathrm{d}x = 0.\]
    Again, we use Fubini theorem to interchange the order of the integral and invoke Lemma \ref{lm:density} to obtain
    \[\int\limits_{\Omega_1}[a^2 \Delta u(y) + \langle \nabla u(y), 
    2\nabla a^2 - a^1(y) \rangle + \left(\Delta a^2 - \operatorname{div}(a^1) + a^0(y)\right)u(y)]v(y) \, \mathrm{d}y = 0.\]
This gives $a^2=0, 2\nabla a^2 - a^1= 0,$ and $\Delta a^2 -\operatorname{div}a^1 + a^0= 0$ locally at each point of $\partial \Omega_1 \cap \Omega$ using the local uniqueness result from Proposition \ref{lm: local uniqueness} replacing $\Omega$ by $\Omega_1$. Now, for each $x \in \partial \Theta_\epsilon(t_0) \setminus B_{\epsilon_0}(x_0)$, $\Omega_1$ can be chosen such that $\partial \Omega_1 \cap \partial \Theta_\epsilon(t_0) = \{x\}$. Thus, $a^0, a^1$ and $a^2$ vanish locally at each point of $\partial \Theta_\epsilon(t_0)$. This proves that $I$ is open. Since $I$ is a non-empty connected set that is both open and closed. This implies $I=[0,1]$. Since $x_1$ was an arbitrary point, this implies $a^j=0$ everywhere in $\Omega$ for $j=0,1,2$. This completes the proof.
\end{proof}

\appendix

\section{Linearization and derivation of the integral identity}\label{appen_linearization}
In the section, we linearize the Dirichlet-to-Neumann map by computing its Fr\'echet derivative, when both Dirichlet and Neumann data are known only on a non-empty open set of the boundary. 
Let $\Sigma \subset \PD \O$ be non-empty and open and let $H^s_\Sigma (\PD \O)$ denote the space of functions $f \in H^s(\PD \O)$ such that $\mathrm{supp} f \subset \Sigma$. Recall the operator \eqref{operator} together with its Dirichlet boundary conditions:
\begin{equation}\label{operator1}
\begin{aligned}
\begin{cases}
    	\Lc_Q(x,D)&=
		(-\Delta)^2 + Q(x,D)\quad  \mbox{in} \quad \Omega
	   \\
	    (u,\PD_{\nu} u)&= (f_1,f_2) \quad \qquad \hspace{9mm} \mbox{on} \quad \PD\Omega
\end{cases}
\end{aligned}
\end{equation}
where $ Q(x,D):= \sum_{l=0}^{3} a^{l}_{i_1\cdots i_{l}}(x) \, D^{i_1\cdots i_l}$	is a differential operator of order $3$ with $1\le i_1,\cdots,i_l\le n$ and $a^{l}$ is a smooth symmetric  tensor field of order $l$ in $\overline{\Omega}$ and $(f_0, f_1) \in H^{7/2}_\Sigma(\PD \O) \times H^{5/2}_\Sigma (\PD \O)$. In what follows, we identify the tensor fields $(a^l)_{l=0}^3$ to the associated differential operator $\sum_{l=0}^{3} a^{l}_{i_1\cdots i_{l}}(x) \, D^{i_1\cdots i_l}$ and let $\mathcal{S}$ denote the space of all such bounded smooth tensor fields, equipped with the $L^{\infty}$ norm. Then, the partial DN map is defined as:
\begin{align}
   \Lambda : \mathcal{S} \to B \left (H^{7/2}_\Sigma(\PD \O) \times H^{5/2}_\Sigma (\PD \O), \big ( H^{3/2}(\PD \O) \times H^{1/2} (\PD \O)\big )\big|_{\Sigma} \right ),\quad \Lambda_{Q} (f_1,f_2):= (\PD^2_{\nu} u|_{\Sigma}, \PD^3_{\nu} u|_{\Sigma}).
\end{align}
where $\big ( H^{3/2}(\PD \O) \times H^{1/2} (\PD \O)\big )\big|_{\Sigma}$ denotes the restriction of the corresponding functions to $\Sigma$. The following lemma gives an expression for the Fr\'echet  derivative of the map $Q\mapsto \Lambda_Q$.

\begin{lemma}
    Suppose $0$ is not an eigen value of \eqref{operator1} and let $P_Q: (H^{7/2}_\Sigma(\PD \O) \times H^{5/2}_\Sigma (\PD \O)) \to H^4(\O) $ denote the solution operator for the problem \eqref{operator1}. Let $G_Q: L^2(\O) \to \Dc(\Lc_Q)$ be the Green operator satisfying
    \begin{equation}
        \begin{aligned}
            \begin{cases}
                \Lc_Q (G_Q F) &= F \quad \mbox{in} \quad \O \\
                (G_Q F, \PD_\nu (G_Q F)) &= 0 \quad \mbox{on} \quad \PD \O.
            \end{cases}
        \end{aligned}
    \end{equation}
    The Fr\'echet derivative of $\Lambda$ is given as
    \begin{align*}
        (\D \Lambda)_Q = B_Q: \Sc \to B \left (H^{7/2}_\Sigma(\PD \O) \times H^{5/2}_\Sigma (\PD \O), \big ( H^{3/2}(\PD \O) \times H^{1/2} (\PD \O)\big )\big|_{\Sigma} \right )
    \end{align*}
    defined for $H \in \Sc$ as 
    \begin{align}
        (B_Q H)(f) &= \bigg ( \PD_\nu^2 G_Q (-HP_Q f), \PD_\nu^3 G_Q(-HP_Q f) \bigg ) \bigg |_{\Sigma}, \quad \mbox{for} \quad f=(f_0,f_1) \in (H^{7/2}_\Sigma(\PD \O) \times H^{5/2}_\Sigma (\PD \O)).
    \end{align}
\end{lemma}
\begin{proof}
    Let $\|H\|_{L^\infty(\Omega)}$ be small enough such that $\Lambda_{Q + H}$ is well defined, where $H$ lives on the same space as $Q$. Given $f = (f_0, f_1) \in H_{\Sigma}^{7/2}(\partial \Omega) \times H_{\Sigma}^{5/2}(\partial \Omega)$, we have
    \begin{align*}
    \Lambda_{Q+H}f-\Lambda_Q f=
    \bigg(\PD^2_{\nu}(P_{Q + H}f - P_Qf), \PD^3_{\nu}(P_{Q + H}f - P_Qf)\bigg) \bigg|_{\Sigma}.
\end{align*}
The function $ w\coloneqq P_{Q+H}f-P_{Q}f$ satisfies the following partial differential equation.
\begin{align*}
    ((-\Delta)^2 +Q(x,D))w &= -H\,w-H\,P_{Q}f \quad  \mbox{in}\quad \Omega\\  (w,\PD_{\nu} w)&=0 \quad \quad \quad \quad \quad \quad \quad\mbox{on} \quad \PD\Omega.
\end{align*}
\newcommand{\nrm}[1]{\lVert #1 \rVert}
Thus, we can write $w=G_{Q} (-H\,w)+ G_{Q}(-H P_{Q}f)$ and $ w\in \mathcal{D}(\Lc_{Q + H})$.
Utilizing the continuity of the Green operator $  G_{Q}(H\,w) $  we obtain
\begin{align*}
    \lVert G_{Q}(H\,w)\rVert_{H^{4}(\Omega)} 
    &\le c \lVert H\,w\rVert_{L^2(\Omega)}\le \frac{1}{2} \lVert w\rVert_{H^{4}(\Omega)}  \quad \mbox{if} \quad \lVert H\rVert_{L^{\infty}(\Omega)} \,\,\, \mbox{is small}. 
\end{align*}
We next estimate $ \lVert w\rv_{H^{4}(\Omega)} = \nrm {G_{Q} (-H\,w)+ G_{Q}(-H P_{Q}f)}_{H^{4}(\Omega)}$. The combination of this  with triangle inequality and last displayed relation implies $\lVert w\rv_{H^{4}(\Omega)} \le \lv H\rv_{L^{\infty}}\,\lv f\rv_{H^{\frac{7}{2}}(\partial\Omega) \times H^{\frac{5}{2}}(\partial\Omega)}.  $ Next we observe that,
\begin{align*}
  & \bigg(\Lambda_{Q+H}(f)-\Lambda_Q(f) - \left ( \PD_\nu^2 G_Q (-HP_Q f), \PD_\nu^3 G_Q(-HP_Q f)\right)|_{\Sigma} \bigg) \\&\,=  \bigg(\PD^2_{\nu} ( w- G_{Q}HP_{Q}f), \PD^3_{\nu} ( w- G_{Q}HP_{Q}f)\bigg) \bigg |_{\Sigma}= \bigg(\PD^2_{\nu} G_{Q}(-Hw), \PD^3_{\nu}G_{Q}(-Hw)\bigg) \bigg |_{\Sigma}.
\end{align*}
Combining trace theorem, continuity of $G_{Q}$ and $ \lVert w\rv_{H^{4}(\Omega)} \le \lv H\rv_{L^{\infty}(\Omega)}\,\lv f\rv_{H^{\frac{7}{2}}(\partial\Omega) \times H^{\frac{5}{2}}(\partial\Omega)}$  we obtain from above 
\begin{align*}
& \lVert (\PD^2_{\nu} G_{Q}(-Hw), \PD^3_{\nu}G_{Q}(-Hw))\rVert_{H^{\frac{3}{2}}(\Sigma) \times H^{\frac{1}{2}}(\Sigma)} \le  \lVert (\PD^2_{\nu} G_{Q}(-Hw), \PD^3_{\nu}G_{Q}(-Hw))\rVert_{H^{\frac{3}{2}}(\partial\Omega) \times H^{\frac{1}{2}}(\partial\Omega)}\\& \le   \lVert G_{Q}(H\,w)\rVert_{H^{4}(\Omega)}   \le \lv H\rv_{L^{\infty}(\Omega)}\, \lVert w\rv_{H^{4}(\Omega)} \le  \lv H\rv^2_{L^{\infty}(\Omega)}\, \lv f\rv_{H^{\frac{7}{2}}(\partial\Omega) \times H^{\frac{5}{2}}(\partial\Omega)}.
\end{align*}
This proves that the Fr{\'e}chet derivative of $Q\mapsto \Lambda_{Q}$ at $Q$ is $B_Q$.
\end{proof}

We are interested in studying the injectivity of $\mathrm{d}\Lambda|_{Q=0}$. This reduces to

\begin{lemma}
    Let $\D \Lambda|_{Q=0} = 0$. Then for any $H= (a^{(l)})_{l=0}^3 \in \Sc$, the following integral identity holds
    \begin{align}
        \int\limits_{\O} \sum\limits_{l=0}^3 a^{(l)}_{i_1 \dots i_l} \PD_{i_1 \dots i_l} u v \, \D x &= 0,
    \end{align}
    for all biharmonic functions $u$ and $v$ in $\O$ whose Dirichlet data is supported in $\Sigma$.
\end{lemma}

\begin{proof}
    Let $g = (g_0, g_1) \in H^{7/2}_\Sigma (\PD \O) \times H^{5/2}_{\Sigma}(\PD \O)$. The function $P_0 g \in H^4(\O)$ satisfies
    \begin{align}
        \begin{cases}
            (-\Delta)^2 P_0 g = 0 &\quad \mbox{in} \quad \O \\
            (P_0 g, \PD_\nu P_0 g) = g &\quad \mbox{on} \quad \PD \O.
        \end{cases}
    \end{align}
Also, for $f = (f_0, f_1) \in H^{7/2}_\Sigma (\PD \O) \times H^{5/2}_{\Sigma}(\PD \O)$,  $G_0 (-HP_0 f)$  solves 
    \begin{align}
        \begin{cases}
            (-\Delta)^2 G_0 (-HP_0 f) = -HP_0 f &\quad \mbox{in} \quad \O \\
            (G_0 (-HP_0 f), \PD_\nu G_0 (-HP_0 f)) = 0 &\quad \mbox{on} \quad \PD \O.
        \end{cases}
    \end{align}
Multiplying $P_0g$ to the last equation and using Green's identities, we get
    \begin{align}
        - \int\limits_{\O} (HP_0 f) P_0g \, \D x & = \int\limits_{\O} (-\Delta)^2 G_0 (-HP_0f) P_0 g - G_0 (-HP_0 f) (-\Delta)^2 P_0g \\
        &= \int\limits_{\PD \O} \bigg[ -\PD_\nu P_0 g (\Delta G_0 (-HP_0f)) + P_0 g \PD_\nu (\Delta G_0 (-HP_0 f)) \\
        & \qquad + \Delta(P_0 g) \PD_\nu (G_0 (-H P_0f)) - G_0(-HP_0f) \PD_\nu(\Delta P_0g) \bigg] \D S.
    \end{align}
The last two terms vanish due to the properties of the Green's function on the boundary. By definition of $P_0$,  we have $(P_0 g, \PD_\nu P_0 g)|_{\PD \O} = (g_0, g_1)$ which vanishes outside $\Sigma$. This yields 
    \begin{align}
        - \int\limits_{\O} (HP_0 f) P_0g \, \D x
        &= \int\limits_{\Sigma} \bigg[ -\PD_\nu P_0 g (\Delta G_0 (-HP_0f)) + P_0 g \PD_\nu (\Delta G_0 (-HP_0 f))\bigg] \D S.
    \end{align}
    Since $\D \Lambda|_{0} = 0$, using the expression for $\D \Lambda|_{0}$ from the previous lemma,   we have $$\left ( \PD_\nu^2 (G_0 (-HP_0f)), \PD_\nu^3 (G_0 (-HP_0f)) \right ) \big |_{\Sigma} = 0.$$
    Now using the equivalence  of  $  \left ( u, \PD_\nu u, \PD_\nu^2 u, \PD_\nu^3 u \right )|_{\Sigma}= 0 \iff \left ( u, \PD_\nu u, \Delta u, \PD_\nu (\Delta u)\right )|_{\Sigma} = 0$, where $\Sigma \subset \PD\Omega$.
 This completes the proof.
\end{proof}
\section{Decay estimates}\label{appendix_decay}
\begin{lemma}\label{lem:decay}
Let $\err_2$ solves the following Dirichlet boundary value problem
\begin{align*}
     \Delta^2 \err_2 &= 0  \,\,\; \,\qquad\quad  \quad\text{in}\quad \Omega \\
   ( \err_2,\PD_{\nu} \err_2) &=  (w_0,\PD_{\nu} w_0) \quad \mbox{on} \quad \PD \Omega.
\end{align*}
Where $ w_0= \chi(x) e^{\frac{-ix\cdot \xi}{h}}$ with $ \chi\in C_c^{\infty}(\Rn)$ and $\xi \in \mathbb{C}^n$ are from Lemma \ref{lm: special solutions}. Then there exists a constant $C$ independent of $\xi$ and $h$ such that 
\begin{align*}
\|\err_2\|_{H^2(\Omega)}  \leq C (1 + \frac{|\xi|^2}{h^2} + \frac{|\xi|^4}{h^4})^{1/2} e^{\frac{1}{h} H_K (\mathrm{Im} \ \xi)}, \quad \mbox{where}  \quad H_K(y)= \sup\limits_{x \in  K} x \cdot y, \quad y\in \mathbb{R}^n.
\end{align*}
    \end{lemma}
    \begin{proof}
From the well-posedness of the above Dirichlet problem \cite{Polyharmonic_Book}*{Theorem 2.20}, we obtain
\[\|\err_2\|_{H^2(\Omega)}  \leq C (\|w_0\|_{H^{3/2}(\partial \Omega)} + \|\frac{\partial w_0}{\partial \nu}\|_{H^{1/2}(\partial \Omega)}),\]
where $w_0(x) = \chi(x) e^{\frac{-ix\cdot \xi}{h}}$. 
Since $\partial \Omega$ is  a compact Riemannian manifold of dimension $n - 1$, we can choose a finite number of coordinate neighborhood system $\{(U_i, \phi_i)\}_{i= 1}^{m_0}$, where
\[\phi_i: U_i  \to V_i \subset \subset \mathbb{R}^{n - 1}\]
is a diffeomorphism from $U_i$ onto an open subset $V_i$ contained in $\mathbb{R}^{n - 1}$. Let $\{\psi_i\}_{i = 1}^{m_0}$ be a partition of unity subordinate to $\{U_i\}_{i = 1}^{m_0}$. The $H^s$-norm of $g \in C^{\infty}(\partial \Omega)$, for $s \in \mathbb{R}$, by
\[ {||g||}_s^2 = \sum_{i = 1}^{m_0}{||(\psi_i g) \circ \phi_i^{-1}||_s^2}.\]
Let $w_{0, j}(x) = (\psi_j w_0) \circ \phi_j^{-1} = \psi_j \circ \phi_j^{-1} \cdot w_0 \circ \phi_j^{-1}$, $1 \leq j \leq m_0 $.  The $H^{3/2}(\PD\Omega)$ norm of $ w$ is defined as: $ ||w||^2_{H^{3/2}(\PD \Omega)}:= \sum_{j=1}^{m_0}||w_{0,j}||_{3/2}^2$, 
where \begin{align*}
    ||w_{0,j}||_{3/2}^2 := ||w_{0,j}||^2 + ||\nabla w_{0,j}||^2 + \int\limits_{V_i}\int\limits_{V_i} \frac{|\nabla w_{0,j}(x) - \nabla w_{0,j}(y)|^2}{|x-y|^n} \, \mathrm{d}x \,\mathrm{d}y.
\end{align*}
It is enough to estimate $ ||w_{0,j}||_{3/2}^2 $. We will  estimate $||w_{0,j}||_{3/2}^2$  component-wise. To this end, we consider the  $ L^2$ norm of $ w_{0,j}$ and observe that 
\begin{align}\label{eq_18}
    ||w_{0,j}||^2 \leq C \int\limits_{V_j}|w_0 \circ \phi_j^{-1}|^2 \, \mathrm{d}x \leq C\sup_{\partial \Omega \cap \operatorname{supp} \chi}|w_{0,j}|^2 \leq C e^{\frac{2}{h}H_K(\operatorname{Im\xi)}}.
\end{align}
Since $\nabla w_{0,j} = \nabla (\psi_j \circ \phi_j^{-1}) \cdot w_0 \circ \phi_j^{-1} + \psi_j \circ \phi_j^{-1} \cdot \nabla(w_0 \circ \phi_j^{-1}) = \nabla (\psi_j \circ \phi_j^{-1}) \cdot w_0 \circ \phi_j^{-1} + \psi_j \circ \phi_j^{-1} \cdot \nabla(w_0) \circ \phi_j^{-1} \cdot J(\phi_j^{-1})$, this implies
\begin{align}\label{eq_19}
    ||\nabla w_{0,j}||^2 \leq C \left( 1 + \frac{|\xi|^2}{h^2} \right) e^{\frac{2}{h}H_K(\operatorname{Im\xi)}}.
\end{align}
Next, we consider
\begin{equation*}
    \begin{split}
        &\int\limits_{V_i}\int\limits_{V_i} \frac{|\nabla w_{0,j}(x) - \nabla w_{0,j}(y)|^2}{|x-y|^n} \\&  \leq
 2  \int\limits_{V_i} \int\limits_{V_i}  \frac{| \nabla (\psi_j \circ \phi_j^{-1})(x) w_0 \circ \phi_j^{-1}(x) -  \nabla (\psi_j \circ \phi_j^{-1})(x) w_0 \circ \phi_j^{-1}(y)|^2}{|x -y|^{n}} \\&\qquad
+ \int\limits_{V_i}\int\limits_{V_i} \frac{|\psi_j \circ \phi_j^{-1}(x) \cdot \nabla(w_0 \circ \phi_j^{-1})(x) - \psi_j \circ \phi_j^{-1}(y) \cdot \nabla(w_0 \circ \phi_j^{-1})(y)|^2}{|x-y|^n}.
    \end{split}
\end{equation*}
Now using mean value theorem on each integrand of the above right hand side inequality and then using the property of $w_0$, we obtain
\begin{equation*}
    \int\limits_{V_i}\int\limits_{V_i} \frac{|\nabla w_{0,j}(x) - \nabla w_{0,j}(y)|^2}{|x-y|^n} \, \mathrm{d}x \mathrm{d}y \leq C \left(1 
+\frac{|\xi|^2}{h^2} + \frac{|\xi|^4}{h^4}\right) e^{\frac{2}{h}H_K(\operatorname{Im}{\xi})}.
\end{equation*}
The combination of this along with \eqref{eq_18} and \eqref{eq_19} gives
\begin{equation*}
||w_{0,j}||_{3/2}^2 \leq C \left(1 
+\frac{|\xi|^2}{h^2} + \frac{|\xi|^4}{h^4}\right) e^{\frac{2}{h}H_K(\operatorname{Im}{\xi})}, \quad \mbox{for each $1 \leq j \leq m_0$}.
\end{equation*}
 This further entails
\begin{equation*}
    ||w_0||_{H^{3/2}{(\partial \Omega)}} \leq C\left(1 
+\frac{|\xi|^2}{h^2} + \frac{|\xi|^4}{h^4}\right)^{1/2} e^{\frac{1}{h}H_K(\operatorname{Im}{\xi})}.
\end{equation*}
In a very similar fashion, one can also obtain 
\[\|{\frac{\partial w_0}{\partial \nu}}\|_{H^{1/2}(\partial \Omega)}
 \leq C\left(1 
+\frac{|\xi|^2}{h^2} + \frac{|\xi|^4}{h^4}\right)^{1/2} e^{\frac{1}{h}H_K(\operatorname{Im}{\xi})}.\]
We next combine preceding estimates and conclude
\[
\|\err_2\|_{H^2(\Omega)}  \leq C (\|w_0\|_{H^{3/2}(\partial \Omega)} + \|\frac{\partial w_0}{\partial \nu}\|_{H^{1/2}(\partial \Omega)}) \\
 \leq C (1 + \frac{|\xi|^2}{h^2} + \frac{|\xi|^4}{h^4})^{1/2} e^{\frac{1}{h} H_K (\mathrm{Im} \ \xi)},
\]
for all $\xi \in \mathbb{C}^n$ such that  $\xi \cdot \xi = 0$, and the constant $C$ is independent of $\xi$.
\end{proof}

\section*{Acknowledgements}
S. K. S was  partly supported by the Academy of Finland (Centre of Excellence
in Inverse Modelling and Imaging, grant 284715) and by the European Research Council
under Horizon 2020 (ERC CoG 770924). S. K. S would like to express his gratitude to M. Salo for their fruitful conversation on this subject. The authors thank Venky Krishnan and Sivaguru Ravisankar for several fruitful discussions which led to the improvement of the manuscript.


\bibliographystyle{siam}
\bibliography{references}

\begin{bibdiv}
\begin{biblist}

\bib{BG_19}{article}{
      author={Bhattacharyya, Sombuddha},
      author={Ghosh, Tuhin},
       title={Inverse boundary value problem of determining up to a second
  order tensor appear in the lower order perturbation of a polyharmonic
  operator},
        date={2019},
        ISSN={1069-5869},
     journal={J. Fourier Anal. Appl.},
      volume={25},
      number={3},
       pages={661\ndash 683},
         url={https://doi.org/10.1007/s00041-018-9625-3},
      review={\MR{3953481}},
}

\bib{BG_biharmonic_second_order}{article}{
      author={Bhattacharyya, Sombuddha},
      author={Ghosh, Tuhin},
       title={An inverse problem on determining second order symmetric tensor
  for perturbed biharmonic operator},
        date={2022},
        ISSN={0025-5831},
     journal={Math. Ann.},
      volume={384},
      number={1-2},
       pages={457\ndash 489},
         url={https://doi.org/10.1007/s00208-021-02276-6},
      review={\MR{4476229}},
}

\bib{polyharmonic_mrt_application}{misc}{
      author={Bhattacharyya, Sombuddha},
      author={Krishnan, Venkateswaran~P.},
      author={Sahoo, Suman~Kumar},
       title={Unique determination of anisotropic perturbations of a
  polyharmonic operator from partial boundary data},
        date={2021},
        note={https://arxiv.org/abs/2111.07610},
}

\bib{BUK}{article}{
      author={Bukhgeim, Alexander~L.},
      author={Uhlmann, Gunther},
       title={Recovering a potential from partial {C}auchy data},
        date={2002},
        ISSN={0360-5302},
     journal={Comm. Partial Differential Equations},
      volume={27},
      number={3-4},
       pages={653\ndash 668},
         url={https://doi.org/10.1081/PDE-120002868},
      review={\MR{1900557}},
}

\bib{Calderon1980}{incollection}{
      author={Calder\'{o}n, Alberto-P.},
       title={On an inverse boundary value problem},
        date={1980},
   booktitle={Seminar on {N}umerical {A}nalysis and its {A}pplications to
  {C}ontinuum {P}hysics ({R}io de {J}aneiro, 1980)},
   publisher={Soc. Brasil. Mat., Rio de Janeiro},
       pages={65\ndash 73},
      review={\MR{590275}},
}

\bib{Krupchyk_isotropic_quasilinear}{article}{
      author={C\^{a}rstea, C\u{a}t\u{a}lin~I.},
      author={Feizmohammadi, Ali},
      author={Kian, Yavar},
      author={Krupchyk, Katya},
      author={Uhlmann, Gunther},
       title={The {C}alder\'{o}n inverse problem for isotropic quasilinear
  conductivities},
        date={2021},
        ISSN={0001-8708},
     journal={Adv. Math.},
      volume={391},
       pages={Paper No. 107956, 31},
         url={https://doi.org/10.1016/j.aim.2021.107956},
      review={\MR{4300916}},
}

\bib{DKSU}{article}{
      author={Dos Santos~Ferreira, David},
      author={Kenig, Carlos~E.},
      author={Salo, Mikko},
      author={Uhlmann, Gunther},
       title={Limiting {C}arleman weights and anisotropic inverse problems},
        date={2009},
        ISSN={0020-9910},
     journal={Invent. Math.},
      volume={178},
      number={1},
       pages={119\ndash 171},
         url={https://doi.org/10.1007/s00222-009-0196-4},
      review={\MR{2534094}},
}

\bib{linearized_partial_data}{article}{
      author={Dos Santos~Ferreira, David},
      author={Kenig, Carlos~E.},
      author={Sj\"{o}strand, Johannes},
      author={Uhlmann, Gunther},
       title={On the linearized local {C}alder\'{o}n problem},
        date={2009},
        ISSN={1073-2780},
     journal={Math. Res. Lett.},
      volume={16},
      number={6},
       pages={955\ndash 970},
         url={https://doi.org/10.4310/MRL.2009.v16.n6.a4},
      review={\MR{2576684}},
}

\bib{Polyharmonic_Book}{book}{
      author={Gazzola, Filippo},
      author={Grunau, Hans-Christoph},
      author={Sweers, Guido},
       title={Polyharmonic boundary value problems},
      series={Lecture Notes in Mathematics},
   publisher={Springer-Verlag, Berlin},
        date={2010},
      volume={1991},
        ISBN={978-3-642-12244-6},
         url={http://dx.doi.org/10.1007/978-3-642-12245-3},
        note={Positivity preserving and nonlinear higher order elliptic
  equations in bounded domains},
      review={\MR{2667016 (2011h:35001)}},
}

\bib{Ghosh-Krishnan}{article}{
      author={Ghosh, Tuhin},
      author={Krishnan, Venkateswaran~P.},
       title={Determination of lower order perturbations of the polyharmonic
  operator from partial boundary data},
        date={2016},
        ISSN={0003-6811},
     journal={Appl. Anal.},
      volume={95},
      number={11},
       pages={2444\ndash 2463},
         url={https://doi.org/10.1080/00036811.2015.1092522},
      review={\MR{3546596}},
}

\bib{Uhlman_Greanleaf_twoplane}{article}{
      author={Greenleaf, Allan},
      author={Uhlmann, Gunther},
       title={Local uniqueness for the {D}irichlet-to-{N}eumann map via the
  two-plane transform},
        date={2001},
        ISSN={0012-7094},
     journal={Duke Math. J.},
      volume={108},
      number={3},
       pages={599\ndash 617},
         url={https://doi.org/10.1215/S0012-7094-01-10837-5},
      review={\MR{1838663}},
}

\bib{Hintz_cpde}{article}{
      author={Hintz, Peter},
      author={Uhlmann, Gunther},
      author={Zhai, Jian},
       title={The {D}irichlet-to-{N}eumann map for a semilinear wave equation
  on {L}orentzian manifolds},
        date={2022},
        ISSN={0360-5302},
     journal={Comm. Partial Differential Equations},
      volume={47},
      number={12},
       pages={2363\ndash 2400},
         url={https://doi.org/10.1080/03605302.2022.2122837},
      review={\MR{4526896}},
}

\bib{Isakov_partial}{article}{
      author={Isakov, V.},
       title={On uniqueness in the inverse conductivity problem with local
  data},
        date={2007},
     journal={Inverse Probl. Imaging},
      volume={1},
       pages={95\ndash 105},
}

\bib{Imanuvilov_JAMS_partial}{article}{
      author={Imanuvilov, Oleg~Yu.},
      author={Uhlmann, Gunther},
      author={Yamamoto, Masahiro},
       title={The {C}alder\'{o}n problem with partial data in two dimensions},
        date={2010},
        ISSN={0894-0347},
     journal={J. Amer. Math. Soc.},
      volume={23},
      number={3},
       pages={655\ndash 691},
         url={https://doi.org/10.1090/S0894-0347-10-00656-9},
      review={\MR{2629983}},
}

\bib{Kian_quasilinear_partial}{article}{
      author={Kian, Yavar},
      author={Krupchyk, Katya},
      author={Uhlmann, Gunther},
       title={Partial data inverse problems for quasilinear conductivity
  equations},
        date={2023},
        ISSN={0025-5831,1432-1807},
     journal={Math. Ann.},
      volume={385},
      number={3-4},
       pages={1611\ndash 1638},
         url={https://doi.org/10.1007/s00208-022-02367-y},
      review={\MR{4566701}},
}

\bib{KLUO_duke}{article}{
      author={Kurylev, Yaroslav},
      author={Lassas, Matti},
      author={Oksanen, Lauri},
      author={Uhlmann, Gunther},
       title={Inverse problem for {E}instein-scalar field equations},
        date={2022},
        ISSN={0012-7094},
     journal={Duke Math. J.},
      volume={171},
      number={16},
       pages={3215\ndash 3282},
         url={https://doi.org/10.1215/00127094-2022-0064},
      review={\MR{4505359}},
}

\bib{KRU2}{article}{
      author={Krupchyk, Katsiaryna},
      author={Lassas, Matti},
      author={Uhlmann, Gunther},
       title={Determining a first order perturbation of the biharmonic operator
  by partial boundary measurements},
        date={2012},
        ISSN={0022-1236},
     journal={J. Funct. Anal.},
      volume={262},
      number={4},
       pages={1781\ndash 1801},
         url={https://doi.org/10.1016/j.jfa.2011.11.021},
      review={\MR{2873860}},
}

\bib{KRU1}{article}{
      author={Krupchyk, Katsiaryna},
      author={Lassas, Matti},
      author={Uhlmann, Gunther},
       title={Inverse boundary value problems for the perturbed polyharmonic
  operator},
        date={2014},
        ISSN={0002-9947},
     journal={Trans. Amer. Math. Soc.},
      volume={366},
      number={1},
       pages={95\ndash 112},
         url={https://doi.org/10.1090/S0002-9947-2013-05713-3},
      review={\MR{3118392}},
}

\bib{KLU_invention}{article}{
      author={Kurylev, Yaroslav},
      author={Lassas, Matti},
      author={Uhlmann, Gunther},
       title={Inverse problems for {L}orentzian manifolds and non-linear
  hyperbolic equations},
        date={2018},
        ISSN={0020-9910},
     journal={Invent. Math.},
      volume={212},
      number={3},
       pages={781\ndash 857},
         url={https://doi.org/10.1007/s00222-017-0780-y},
      review={\MR{3802298}},
}

\bib{K_S}{article}{
      author={Kenig, Carlos},
      author={Salo, Mikko},
       title={The {C}alder\'{o}n problem with partial data on manifolds and
  applications},
        date={2013},
        ISSN={2157-5045},
     journal={Anal. PDE},
      volume={6},
      number={8},
       pages={2003\ndash 2048},
         url={https://doi.org/10.2140/apde.2013.6.2003},
      review={\MR{3198591}},
}

\bib{Kenig_Salo_Survey}{incollection}{
      author={Kenig, Carlos},
      author={Salo, Mikko},
       title={Recent progress in the {C}alder\'on problem with partial data},
        date={2014},
   booktitle={Inverse problems and applications},
      series={Contemp. Math.},
      volume={615},
   publisher={Amer. Math. Soc., Providence, RI},
       pages={193\ndash 222},
         url={http://dx.doi.org/10.1090/conm/615/12245},
      review={\MR{3221605}},
}

\bib{Kenig_annals_2007}{article}{
      author={Kenig, Carlos~E.},
      author={Sj\"{o}strand, Johannes},
      author={Uhlmann, Gunther},
       title={The {C}alder\'{o}n problem with partial data},
        date={2007},
        ISSN={0003-486X},
     journal={Ann. of Math. (2)},
      volume={165},
      number={2},
       pages={567\ndash 591},
         url={https://doi.org/10.4007/annals.2007.165.567},
      review={\MR{2299741}},
}

\bib{Krupchyk_remark}{article}{
      author={Krupchyk, Katya},
      author={Uhlmann, Gunther},
       title={A remark on partial data inverse problems for semilinear elliptic
  equations},
        date={2020},
        ISSN={0002-9939},
     journal={Proc. Amer. Math. Soc.},
      volume={148},
      number={2},
       pages={681\ndash 685},
         url={https://doi.org/10.1090/proc/14844},
      review={\MR{4052205}},
}

\bib{LLLS_JMPA}{article}{
      author={Lassas, Matti},
      author={Liimatainen, Tony},
      author={Lin, Yi-Hsuan},
      author={Salo, Mikko},
       title={Inverse problems for elliptic equations with power type
  nonlinearities},
        date={2021},
        ISSN={0021-7824},
     journal={J. Math. Pures Appl. (9)},
      volume={145},
       pages={44\ndash 82},
         url={https://doi.org/10.1016/j.matpur.2020.11.006},
      review={\MR{4188325}},
}

\bib{Fractional_power_LLST}{article}{
      author={Liimatainen, Tony},
      author={Lin, Yi-Hsuan},
      author={Salo, Mikko},
      author={Tyni, Teemu},
       title={Inverse problems for elliptic equations with fractional power
  type nonlinearities},
        date={2022},
        ISSN={0022-0396},
     journal={J. Differential Equations},
      volume={306},
       pages={189\ndash 219},
         url={https://doi.org/10.1016/j.jde.2021.10.015},
      review={\MR{4332042}},
}

\bib{SS_linearized}{article}{
      author={Sahoo, Suman~Kumar},
      author={Salo, Mikko},
       title={The linearized {C}alder\'{o}n problem for polyharmonic
  operators},
        date={2023},
        ISSN={0022-0396},
     journal={J. Differential Equations},
      volume={360},
       pages={407\ndash 451},
         url={https://doi.org/10.1016/j.jde.2023.03.017},
      review={\MR{4562046}},
}

\bib{SYL}{article}{
      author={Sylvester, J.},
      author={Uhlmann, G.},
       title={A global uniqueness theorem for an inverse boundary value
  problem},
        date={1987},
     journal={Ann. of Math. (2)},
      volume={125},
      number={1},
       pages={152\ndash 169},
}

\bib{Uhl_eip_survey}{article}{
      author={Uhlmann, Gunther},
       title={Electrical impedance tomography and {C}alder\'{o}n's problem},
        date={2009},
        ISSN={0266-5611},
     journal={Inverse Problems},
      volume={25},
      number={12},
       pages={123011, 39},
         url={https://doi.org/10.1088/0266-5611/25/12/123011},
      review={\MR{3460047}},
}

\bib{Uhlmann_survey}{incollection}{
      author={Uhlmann, Gunther},
       title={30 years of {C}alder\'{o}n's problem},
        date={2014},
   booktitle={S\'{e}minaire {L}aurent {S}chwartz---\'{E}quations aux
  d\'{e}riv\'{e}es partielles et applications. {A}nn\'{e}e 2012--2013},
      series={S\'{e}min. \'{E}qu. D\'{e}riv. Partielles},
   publisher={\'{E}cole Polytech., Palaiseau},
       pages={Exp. No. XIII, 25},
      review={\MR{3381003}},
}

\bib{kelvin-biharmonic}{article}{
      author={Xu, X.},
       title={Uniqueness theorem for the entire positive solutions of
  biharmonic equations in {${\bf R}^n$}},
        date={2000},
        ISSN={0308-2105},
     journal={Proc. Roy. Soc. Edinburgh Sect. A},
      volume={130},
      number={3},
       pages={651\ndash 670},
         url={https://doi.org/10.1017/S0308210500000354},
      review={\MR{1769247}},
}

\end{biblist}
\end{bibdiv}

\end{document}